\documentclass{amsart}
\usepackage[utf8]{inputenc}
\usepackage{graphicx,amssymb,latexsym}
\input xy
\xyoption{all}
\usepackage[all]{xy}
\graphicspath{ {./images/} }

\usepackage{color}
\usepackage{pb-diagram}
\usepackage{epstopdf}
\usepackage{multicol}

\newtheorem{thm}{Theorem}[section]

\newtheorem{cor}[thm]{Corollary}
\newtheorem{lem}[thm]{Lemma}
\newtheorem{prop}[thm]{Proposition}
\theoremstyle{definition}
\newtheorem{defn}[thm]{Definition}
\newtheorem{rem}[thm]{Remark}

\numberwithin{equation}{section}
\newcommand{\QQ}{\mathbb Q}

\newcommand{\ZZ}{\mathbb Z}
\newcommand{\CC}{\mathbb C}
\newcommand{\PP}{\mathbb P}

\newcommand{\lra}{\longrightarrow}

\newcommand{\ra}{\rightarrow}

\newcommand{\cA}{\mathcal{A}}
\newcommand{\cL}{\mathcal{L}}

\newcommand{\cU}{\mathcal{U}}

\newcommand{\tC}{\widetilde{C}}

\newcommand{\cE}{\mathcal{E}}

\newcommand{\cO}{\mathcal{O}}
\newcommand{\cD}{\mathcal{D}}

\newcommand{\cR}{\mathcal{R}}

\newcommand{\cZ}{\mathcal{Z}}

\newcommand{\cB}{\mathcal{B}}
\newcommand{\ch}{\mathfrak{h}}

\DeclareMathOperator{\Aut}{{Aut}}

 \DeclareMathOperator{\Ker}{ker}
  \DeclareMathOperator{\Fix}{Fix}
\DeclareMathOperator{\Pic}{Pic}
 \DeclareMathOperator{\Nm}{{Nm}}

 \DeclareMathOperator{\Ima}{{Im}}
 \DeclareMathOperator{\diag}{{diag}}

 \DeclareMathOperator{\fix}{Fix}
 \DeclareMathOperator{\im}{Im}

\DeclareMathOperator{\id}{id}

\newcommand{\s}{\sigma}
\renewcommand{\t}{\tau}
\begin{document}

\title[ ]{Klein coverings of genus 2 curves}
\author{ Pawe\l{} Bor\'owka,  Angela Ortega}
\address{P. Bor\'owka \\ Institute of Mathematics, Jagiellonian University in Krak\'ow, Poland}
\email{Pawel.Borowka@uj.edu.pl}
              
\address{A. Ortega \\ Institut f\" ur Mathematik, Humboldt Universit\"at zu Berlin \\ Germany}
\email{ortega@math.hu-berlin.de}

%\thanks{}
%\keywords{}%

\date{\today }
%\dedicatory{ }%
%\commby{ }%
% ----------------------------------------------------------------

\begin{abstract} 
We investigate the geometry of \'etale $4:1$ coverings of smooth complex genus 2 curves with the monodromy group
isomorphic to the Klein four-group. There are two cases, isotropic and non-isotropic depending on the values
of the Weil pairing restricted to the group defining
the covering. We recall from our previous work \cite{bo} the results concerning the non-isotropic case and
fully describe the isotropic case. We show that the necessary information to construct the Klein coverings is
encoded in the 6 points on $\PP^1$ defining the genus 2 curve.
The main result of the paper is the fact that, in both cases the Prym map associated to these coverings is
injective. Additionally, we provide a concrete description of the closure of the image of the Prym map inside the 
corresponding moduli space of polarised abelian varieties.
\end{abstract}

\maketitle
%\tableofcontents

\section{Introduction}

Classically, Prym varieties are principally polarized abelian varieties arising from \'etale double coverings
$f:C'\ra C$ between smooth curves. The Prym variety of such a covering is the complementary abelian subvariety
to the image of $f^*(JC)$ inside the Jacobian $JC'$.  By extending this to the moduli space of coverings one can
define the Prym map to the moduli space $\cA_g$ of principally polarised abelian varieties (ppav). 
It is known that the Prym map is dominant for $g\leq 5$ (\cite{W}), so the general ppav of dimension $\leq 5$
is a Prym variety.  Mumford revived the theory of Prym varieties by investigating in \cite{M} the relations
between the Jacobian of $C'$ and the Prym variety $P(C'/C)$, which are known as the  Schottky-Jung relations. 
For $g\geq 7$ the Prym map is generically injective but never injective (\cite{B}, \cite{De}, \cite{FS}, \cite{We}). 
The study of the the Prym map for \'etale double coverings of genus 6 curves led to the result that $\cA_5$ is
unirational (\cite{V}). These results inspired the research of Pryms for other coverings, which are no longer
principally polarised. The Torelli problem in the case of Pryms of ramified double coverings  has been fully
solved in a  series of recent papers \cite{NO,MN,MP} and the case of \'etale cyclic coverings has
been covered in \cite{lo}.

The main aim of this paper is to go a step further and study the most basic non-cyclic coverings, namely \'etale
$4:1$ coverings of genus 2 curves with  monodromy group isomorphic
to the Klein four-group. These are examples of bidouble coverings but in the case curves we call
them Klein coverings. Since the group is non-cyclic, it turns out that the space of such coverings have two
disjoint components depending on the values of the Weil pairing restricted to the
group defining the covering, called isotropic and non-isotropic. The non-isotropic case has appeared naturally in
the investigation of smooth hyperelliptic curves that can be embedded into abelian surfaces (\cite{bo}). One of
the results of \cite{bo} is that such curves are non-isotropic Klein coverings of genus 2 curves.
In this paper we show that the corresponding Prym map for non-isotropic Klein coverings over genus 2 curves is
injective (see Theorem \ref{thm:nonisoprym}). 

The isotropic case is quite different from the non-isotropic case since the covering is not longer 
hyperelliptic. The good side of the
story is that we get more non-trivial quotient curves and in particular we have been able to show that the Prym
variety of the covering
is the image of the Jacobian of one of them. This allows us to show that in this case, the Prym map is also
injective (see Theorem \ref{thm:isoprym}).

Putting the two theorems together we get the main result of the paper as follows.
\begin{thm}
Let $H$ be a smooth complex genus 2 curve and $V_4$ be a Klein four-subgroup of the group of two torsion 
points on the Jacobian $JH$. 
The space of Klein coverings is the disjoint union $R_2^{iso}\cup R_2^{ni}$ depending on whether $V_4$ is
isotropic or non-isotropic
with respect to the Weil pairing. Moreover, the Prym maps 
\[ {\Pr}^{iso}:\cR_2^{iso}\lra\cA_3^{(1,2,2)}\]
\[ {\Pr}^{ni}:\cR_2^{ni}\lra \cA_3^{(1,1,4)}\]
are both injective onto their images.
\end{thm}

The idea of the proof is that the Klein covering provides a tower of curves in which one can find elliptic curves that generate
the Prym variety. Then we show that all the information is actually contained in the division of 6 points on $\PP^1$ into 3 pairs in the
isotropic case and 2 triples in the non-isotropic case. This in particular shows that $R_2^{iso}$ and $R_2^{ni}$ are irreducible.

The paper is structured as follows. In the Preliminaries we collect the necessary results on curves with
involutions. An interesting result on its own is the characterisation of non-hyperelliptic curves of genus 3 
admitting 2 commuting involutions
(see Lemma \ref{lem:nonhyp3}) and their Jacobians, (see Proposition \ref{prop:nonhyp3}). Such curves will play crucial role in the isotropic 
Klein coverings.

Section \ref{sec:niso} is devoted to non-isotropic Klein coverings. It consists of a summary of results from \cite{bo} and Theorem \ref{thm:nonisoprym} that states that the Prym map in this case is injective.
Section \ref{sec:iso} explains the construction of isotropic Klein coverings. Here, we characterise the curves
that appear as quotient curves and the covering maps between them. The conclusion is Theorem \ref{thm:isoprym}
that states that the Prym map is injective also in this case. In Section 5 we give a precise
description, in terms of period matrices, of the image of the Prym map in both cases.

\subsection*{Acknowledgements}
We would thank Jennifer Paulhus for showing us the LMF database
which contains examples of families of curves with prescribed 
automorphism groups. We are grateful to Klaus Hulek for suggesting related questions to investigate,
Section 5 is the outcome of one of these.

The work of the first author has been partially supported by the Polish National Science Center project number 2018/02/X/ST1/02301. 
The work of the second author has been partially supported by the National Science Foundation under
Grant N0. DMS-1440140 while the author was in residence at the Mathematical Sciences Research Institute in Berkeley,
California, during the Spring 2019 semester.

\section{Preliminaries}
In the paper we deal with elliptic curves, i.e. genus 1 curve with a chosen point. In this way, an elliptic curve is a one
dimensional abelian variety. Thus, the involution  $-1$ becomes a (hyper)elliptic involution and we define the
Weierstrass points on an elliptic curve as the 2-torsion points, i.e. the fixed points of the involution. 

 A double covering $f:C'\ra C$ can be defined from upstairs
 as the quotient of $C'$ by an involution $\s$ so that  $C=C'/\s$, or from downstairs by fixing a branching 
 divisor $B$ on $C$ and a line bundle $\cL$ such that $\cL^{\otimes2}=\cO_C (B)$. In the paper we use both 
 perspectives. If  $\s$ is the involution that exchanges the sheets of the covering, then we denote 
 by $C_\s$ the quotient curve. Sometimes, to stress that the quotient curve is elliptic, we write 
 $E_\s=C/\s$ instead of $C_\s$.

As in \cite{bo}, we define Klein coverings as follows:
\begin{defn}
A Klein covering of a curve $C$ is a 4:1 \'etale covering $f:\tC\ra C$ with  monodromy group isomorphic to 
the Klein four-group $V_4 = \ZZ_2\times\ZZ_2$. In particular, it is defined by a Klein four-subgroup $G$
of the set of $2$-torsion points on the Jacobian $JC$, denoted by $JC[2]$.   
\end{defn}
In \cite{bo} we have shown that there are two disjoint cases of Klein coverings that depends on the values of the Weil pairing restricted to $G$. 
\begin{defn}
We call a Klein covering isotropic if the group $G \subset JC[2]$ defining the covering is isotropic with respect to the Weil pairing. In the other case, we call the covering non-isotropic.
\end{defn}
\begin{rem}\label{rem:Weil}
Let $H$ be a hyperelliptic curve of genus $g$. Recall from \cite{bo,D} that every $2$-torsion point on $JH$ can be written uniquely as a
sum of an even number of at most $g$ Weierstrass points, so the Weil paring of $\eta_1,\eta_2\in JH[2]$ can be computed as the parity of 
the number of common Weierstrass points in the presentations of $\eta_1,\eta_2$. 
\end{rem}

Through the paper the involution $\iota$ will always denote the hyperelliptic involution.
We recall the basic facts about involutions and their lifts.

\begin{lem}\label{lem:lift}
Let $f:C\ra C_{\t}$ be the double covering given by the 
fixed point free involution $\t$. Let $\rho$ be an involution on $C_\t$
and denote $\fix(\rho)$ the set of fixed points. Let $r,r\t$  be the lifts of $\rho$. Then $\fix(r)$ and
$\fix(r\t)$ are disjoint, their cardinalities are divisible by $4$ and $f^{-1}(\fix(\rho))=\fix(r)\cup\fix(r\t)$ 
is twice the cardinality of $\fix(\rho)$.
\end{lem}
\begin{proof}
Let $C_\t$ be of genus $g$. Then $C$ is of genus $2g+1$. Using Hurwitz formula, for any involution, say $\alpha$ on $C$ we have
$$2(2g+1)-2=2(2g_\alpha-2)+|\fix(\alpha)|.$$
hence $\fix(\alpha)$  has cardinality divisible by 4.

If $x\in\fix(\rho)$ and $f^{-1}(x)=\{P,Q\}$ then either $r(P)=P$ and $r(Q)=Q$ or $r(P)=Q$ and $r(Q)=P$. In the former case, $r\t$ exchange 
points in the fibre and in the latter case $r\t$ fixes them. Hence every preimage of a fixed point of $\rho$ belongs to exactly one fixed point set.  
\end{proof}
As an easy, yet important consequence of the proof, we get the following corollary.
\begin{cor}\label{cor:liftinv}
Let $f:C\ra C_{\t}$ be a double covering. Then any lift of an involution with at least one fixed point that is
not a branch point of $f$ is again an involution.
\end{cor}
All the involutions in the paper whose lift is considered satisfy the condition of Corollary \ref{cor:liftinv},
so in the sequel the lift of an involution under a double covering is again an involution. 

\begin{lem}\label{lem:weier}
Let $W=\{w_1,\ldots,w_{2g+2} \}$ be the set of Weierstrass points on a hyperelliptic curve $H$ of genus $g$. Then for any subset A of 
size $g+1$ we have the following equality in $\Pic^{g+1}(H)$
$$\sum_{w_i\in A}w_i=\sum_{w_j\in W\setminus A}w_j$$
\end{lem}
\begin{proof}
When one considers an equation of a hyperelliptic curve, then all Weierstrass points become roots of a polynomial. This means in
particular that the sum of all Weierstrass points is linear equivalent to $(2g+2)P$ where 
$P$ is the preimage of the infinity. 
Hence 
$$\sum_{w_i\in W}w_i \sim (2g+2)P \ \Rightarrow  \ \sum_{w_i\in A}w_i-(g+1)P \sim (g+1)P-\sum_{w_i\in W\setminus A}w_i$$
Since Weierstrass points are 2-torsion points in the Jacobian of $H$, one gets the result.
\end{proof}

\begin{cor} \label{cor:weier}
For an elliptic curve E one has $w_1+w_2=w_3+w_4$ for $w_i\in E[2]$.

For a genus 2 curve $H$ one has $w_1-w_2+w_3-w_4=w_5-w_6$.
\end{cor}

Recall the following conditions of a double covering to be hyperelliptic:
\begin{prop}[\cite{bo}] \label{farkas-lemma}
Let $H$ be a hyperelliptic curve of genus $g$ and $h:C\ra H$ an \'etale double covering defined by 
$\eta\in JH[2]\setminus \{0\}$.  Then $C$ is hyperelliptic if and only if 
$\eta= \cO_H(w_1-w_2)$, where $w_1, w_2 \in H $ are Weierstrass points.
\end{prop}

\begin{lem}\label{lem:doucovell4p}
A double cover $\pi: C \ra E$ of an elliptic curve branched along 4 points of the  
form $x, \iota_E x, y, \iota_E y $ is hyperelliptic of genus 3 if and only if  the line
bundle defining the covering is the hyperelliptic bundle.
\end{lem}

\begin{proof}
Let $B = x + \iota_E x +  y + \iota_E y$ be the branch divisor of the covering $\pi: C\ra E$
and $\eta \in \Pic^{2}(E)$ be the line bundle defining $\pi$, so $\eta^2\simeq \cO_E(B) = 
\cO_E(2H_E)$, where $H_E$ denotes the (hyper)elliptic divisor.
Suppose that $C$ is hyperelliptic with hyperelliptic divisor $H_C$. 
Then $\pi^* H_E = \pi^* (2w_1) = 2H_C$
%Since the Weierstrass
%points of $C$ map onto the Weierstrass points of $E$, we have $\pi^* H_E = \pi^* (2w_1) = 2H_C$.

By the projection formula 
\begin{eqnarray*}
H^0(C, \cO_C (2H_C) ) &=& H^0(C, \pi^*\cO_E(H_E))\\
 &=& H^0(E, \cO_E(H_E) \otimes \pi_*(\cO_C))\\
&=& H^0(E, \cO_E(H_E)) \oplus H^0(E,\cO_E(H_E)) \otimes \eta^{-1}).
\end{eqnarray*}

Since $h^0(C, \cO_C (2H_C)) = h^0(C, \omega_C) =3$, necessarily 
$h^0(E,\cO_E(H_E)) \otimes \eta^{-1})=1$, hence $\eta \simeq \cO_E(H_E))$.

Conversely, if $\eta = \cO_E(H_E))$ then, again by the projection formula, 
one computes 
$$
h^0(C, \pi^*\cO_E(H_E)) = h^0(E, \cO_E(H_E)) + h^0(E, \cO_E(H_E) \otimes \eta^{-1}) = 3.
$$
According to Clifford's theorem $C$ is hyperelliptic and  $\pi^*\cO_E(H_E)$ a
multiple of the hyperelliptic bundle.

\end{proof}

\subsection{Curves with a $\ZZ_2^2$ group of involutions}
We start by considering involutions on hyperelliptic curves.

\begin{lem}\label{fpf-inv}
Let $X$ be a smooth hyperelliptic curve  of genus $g=2k+1$.  %Assume $Aut(X) \subset \langle \iota, \sigma, \tau\rangle \simeq \ZZ^2_2$, 
with $\iota$ the hyperelliptic involution. For any $\tau \in Aut(X) $, $\tau^2 = \id, \ 
\tau \neq \iota$ we have 
$$
\Fix (\tau) =\emptyset, \  |\Fix(\iota\tau)| =4,   \quad \mbox{or} \quad | \Fix (\tau)| =4, \ \Fix(\iota\tau) =\emptyset.
$$
\end{lem}
\begin{proof}
According to Accola's Theorem (\cite[Theorem 5.9]{A}) applied to the group $G_0=\langle \iota, \tau \rangle$ we have
\begin{equation}\label{Accola}
 2(2k+1) = 2g(X/ \tau ) + 2g(X/ \iota\tau ).   
\end{equation}
By Hurwitz formula one deduces that $g(X/\tau )\leq k+1$ and $g(X/ \iota\tau )\leq k+1$.
Together with \eqref{Accola} this implies $g(X/\tau )= k+1$ and $g(X/ \iota\tau )= k$
or vice versa.
\end{proof}
\begin{rem}
Geometrically, Lemma \ref{fpf-inv} means that if a hyperelliptic curve of genus $g=2k+1$ is a double covering then the covering is \'etale or branched in 4 points. Moreover, if one case occurs then the other, too.  
\end{rem}

\begin{lem}\label{lem:hypgen3}
Let $C$ be a genus 3 curve with two involutions. Then
$C$ is hyperelliptic %with the hyperelliptic involution $\s$ 
if and only if $C$ is an \'etale double cover of a genus 2 curve $H$. 
\end{lem}
\begin{proof}
One implication is the case $k=1$ of Lemma \ref{fpf-inv}. The other follows from Proposition \ref{farkas-lemma}. 
\end{proof}

\begin{lem}\label{lem-fpf}
With the assumptions of Lemma \ref{fpf-inv} and $\sigma, \tau$ fix-point free involutions, then $\sigma\tau$
is also a fix-point-free involution for $k$ an even number.
\end{lem}

\begin{proof}
Set $g_0=g(X/ \langle  \tau, \sigma\tau \rangle)$ and $g_{\alpha}=g(X/ \alpha )$,
for $\alpha  \in Aut(X)$.  We apply Accola's theorem (\cite{A}) to the group $G_0=\langle  \tau, \sigma\tau \rangle$ to
obtain
\begin{eqnarray*}
2(2k+1) + 4g_0 & = & 2g_{\sigma} + 2g_{\tau} + 2g_{\sigma\tau}\\
2k+1 +2g_0 & = &g_{\sigma} + g_{\tau}  + g_{\sigma\tau} \\
 &=& 2k + 2 + g_{\sigma\tau}.
\end{eqnarray*}
Since the left hand side is odd we conclude that $g_{\sigma\tau}$ is odd. By Lemma \ref{fpf-inv} applied to 
$\sigma\tau$ we get that $g_{\sigma\tau} = k+1$ or $g_{\sigma\tau} = k$. Thus, if $k$ is even
$g_{\sigma\tau} = k+1$ and $\sigma\tau$ is fix-point-free.
\end{proof}
\begin{rem}
The fact that $X$ is hyperelliptic in Lemma \ref{lem-fpf} is crucial. Proposition \ref{prop:involutionj} shows an example when involutions $\s,j$ are fixed point free on a non-hyperelliptic genus 5 curve and $j\s$ has 8 fixed points. 
\end{rem}

For non-hyperelliptic curves of genus 3, we get the following results.
\begin{lem}\label{lem:nonhyp3}
Let $\s,\t\in\Aut(C)$ be two commuting involutions on a non-hyperelliptic genus 3 curve $C$.
Then $C$ is a double covering of three elliptic curves, i.e. all 3 involutions are
bielliptic. In this case the sum of points in $\Fix(\s)$ (respectively in $\Fix(\t)$ and $\Fix(\s\t)$) form a
divisor linearly equivalent to $K_C$. Moreover, $C$ is a $4:1$ covering of $\PP^1$ with the 
monodromy group isomorphic to the Klein four-group and the branching divisor being 6 points (images of 12 fixed points).
\end{lem}
\begin{proof}
By Lemma \ref{lem:hypgen3} the quotient curves cannot be of genus 2 and by assumption cannot be of genus 0, so all involutions are 
bielliptic. The second part follows from the fact that the canonical divisor of an elliptic curve is trivial, so the canonical divisor 
of a curve equals the ramification divisor of an elliptic covering, that is the fixed points divisor of an involution. 
By the commutativity of the three involutions, the actions on $\Fix(\s)$ of $\t$ and $\s\t$ have to be equal to each other and equal 
to the pair of transpositions. In particular, the covering $C\ra C/\langle \s,\t \rangle$ has branching divisor at least as stated. 
Hurwitz formula gives $$2\cdot3-2=4\cdot(2\cdot0-2)+\deg(R),$$ hence $\deg(R)=12$ and the branching divisor is exactly as stated. 
\end{proof}

Consider the elliptic curves $E_{\sigma},E_{\tau} $ and $E_{\sigma\tau}$ as subvarieties of $JC$ as
follows
$$
E_{\sigma}= \im (1+\sigma)=\Fix(\sigma), \quad E_{\tau}=\im (1+\tau)=\Fix(\tau), \quad E_{\sigma\tau}
=\im (1+\sigma\tau)=\Fix(\sigma\tau).
$$
Here, we abuse the notation by writing the same letter for the automorphism of a curve and its extension to the Jacobian. We denote by $P(C/E)$ the Prym variety of the covering, for definition see Section \ref{sec:abvar}.
\begin{prop}\label{prop:nonhyp3}
The map $\varphi: E_{\sigma} \times E_{\tau} \times E_{\sigma\tau} \ra JC $ is an isogeny of degree 8.
\end{prop}

\begin{proof}
Since the varieties are of the same dimension it is enough to show that the kernel of 
$\varphi$ is finite and in fact of cardinality 8. First, we consider the addition map
$$
\varphi_{\sigma\tau}: E_{\sigma} \times E_{\tau} \ra P(C / E_{\sigma\tau}), 
\qquad (a,b) \mapsto a+b,
$$
where $P(C / E_{\sigma\tau})$ denotes the Prym variety of the map $C \ra E_{\sigma\tau}$. Let $a_{\alpha}: C \ra  E_{\alpha}$ and $b_{\alpha}:  E_{\alpha} \ra \PP^1$
be the quotient maps  with $\alpha \in  \{ \sigma, \tau, \sigma\tau\}$.
Notice that the maps $a_{\alpha}^*: E_{\alpha} \ra JC$ are injective. 
Using the results in \cite[Appendix]{rr} we have that 
$$
\Ker \varphi_{\sigma\tau} = a^*_{\sigma}(E_{\sigma}[2]) \cap a^*_{\tau}(E_{\tau}[2]). 
$$
We shall give a precise description of this kernel.  Let 
\begin{eqnarray*}
\Fix(\sigma) &=& \{ S_1, S_2, \tau (S_1), \tau (S_2)\} \\
\Fix(\tau) &=& \{ T_1, T_2, \sigma (T_1), \sigma (T_2)\} \\
\Fix(\sigma\tau) &=& \{ R_1, R_2, \tau (R_1), \tau (R_2)\} 
\end{eqnarray*}
be the sets of fixed points in $C$ of the automorphisms $\sigma, \tau, \sigma\tau$ respectively.
Then for $i=1,2$
$$
\{a_{\sigma}(T_i), a_{\sigma}(R_i)\}, \quad \{a_{\tau}(R_i), a_{\tau}(S_i)\},  \quad
\{a_{\sigma\tau}(T_i), a_{\sigma\tau}(S_i)\}
$$
are the ramification points of $b_{\sigma},b_{\tau}, b_{\sigma\tau}$ respectively. Thus
\begin{eqnarray*}
E_{\sigma}[2] & = & \{ 0, [a_{\sigma}(R_1)-a_{\sigma}(T_1)],
[a_{\sigma}(R_1)-a_{\sigma}(T_2)], [a_{\sigma}(R_1)- a_{\sigma}(R_2)]\}  \\
E_{\tau}[2] & = & \{ 0, [a_{\tau}(R_1)-a_{\tau}(S_1)], [a_{\tau}(R_1)- a_{\tau}(S_2)],
[a_{\tau}(R_1)- a_{\tau}(R_2)]\}
\end{eqnarray*}
and analogously,  one can describe the 2-torsion points of $E_{\sigma\tau}$.  From this we clearly have
$$
a_\sigma^*(E_{\sigma}[2]) \cap a_{\tau}^*(E_{\tau}[2]) = \{0, [R_1 + \sigma(R_1) - (R_2 + \sigma(R_2)) ] \},
$$
by observing that $\tau(R_i) =\sigma{R_i}, \ i=1,2$.
Therefore $|\Ker \varphi_{\sigma\tau}| =2$. 

On the other hand, the kernel of the addition map
$$
\psi: P(C/ E_{\sigma\tau}) \times E_{\sigma\tau} \ra JC
$$
is given by the set $\{(c, c) \ \mid \ c \in E_{\sigma\tau}[2] \}$, since  $P(C/E_{\sigma\tau}) \cap E_{\sigma\tau}
= \ker (1+\sigma\tau) \cap E_{\sigma\tau} =E_{\sigma\tau}[2]$. It follows that the kernel of the 
composition map 
$$
E_{\sigma} \times E_{\tau}  \times E_{\sigma\tau} \stackrel{\varphi_{\sigma\tau} \times id}{\lra} 
P(C / E_{\sigma\tau})  \times E_{\sigma\tau} \stackrel{\psi}{\lra} JC
$$
has cardinality equal to $|\ker \varphi_{\sigma\tau} | \cdot | E_{\sigma\tau}[2] | = 8$.
\end{proof}

If $\Theta_{C}$ denotes the principal polarisation on $JC$, one can check that 
$\varphi^* (\Theta_C)$ is algebraically equivalent to twice the principal product
polarisation on  $E_{\sigma} \times E_{\tau} \times E_{\sigma\tau}$.
Using the fact that an embedding of an elliptic curve $E$ in a Jacobian $JC$ is
equivalent to having a covering $C \ra E$, we get the inverse statement:
\begin{cor}
If the Jacobian of a non-hyperelliptic genus 3 curve $C$ is polarised isogenous to the
product of three elliptic curves with twice 
the principal polarisation, then $\Aut(C)$ contains a $\ZZ_2^2$ group of involutions.
\end{cor}
\begin{rem}
Note that to write the proof of Proposition \ref{prop:nonhyp3} we broke the natural symmetry. It is easy to see that 
$S_i+\t(S_i)$, $T_i+\s(T_i)$ and $R_i+\t(R_i)$ are divisors fixed by all three
involutions. Since the intersection of three the 
elliptic curves in $JC$ consists of $0$ and a 2-torsion point, one sees that all the
respective differences are actually equal to each other.
\end{rem}
We summarise  the description of genus 3 curves with the following corollary.
\begin{cor}\label{cor:distpoint}
Let $C$ be a non-hyperelliptic genus 3 curve admitting a $\ZZ_2^3$ subgroup of automorphisms.
Then, its Jacobian contains a distinguished (non-zero) $2$-torsion point lying in the intersection of three
elliptic curves.
\end{cor}
%\begin{proof}
%The intersection of three elliptic curves cannot have more than 2 points, since then the addition map would have degree at least 16. On the other hand, according to Proposition  a point $R_1 + \sigma(R_1) - (R_2 + \sigma(R_2))=S_1 + \sigma(S_1) - (S_2 + \sigma(S_2)) =T_1 + \sigma(T_1) - (T_2 + \sigma(T_2)) $ certainly lies in the intersection of them.
%\end{proof}

\subsection{Some results concerning abelian varieties}\label{sec:abvar}
We start by recalling the definition of a Prym variety. A finite morphism $f:C' \ra C$ between smooth curves induces a
homomorphism of groups
$\Nm_f : JC' \ra JC$ between the Jacobians, called the {\it norm map} defined by pushing down divisors of
degree zero: 
$$
\Nm_f: [\sum_i n_i p_i] \ra [\sum_i n_i f(p_i)].
$$
When $f$ is \'etale the kernel of $\Nm_f$ consists of several connected components, so the Prym variety
 of $f$ is the defined as the connected component of $\Ker \Nm_f$ containing the zero element:
 $$P=P(C'/C):=(\Ker \Nm_f)^0.$$ 
 It is of dimension 
$\dim P = g(C') - g(C)$ (the difference of the genera of the curves) and as an abelian subvariety of $JC'$ has a natural polarization
$\Xi$, induced by the restriction of the principal polarization of $JC'$. In particular, the Prym variety can be seen as a complementary abelian subvariety to the image of $f^*(JC)$ inside $JC'$.
It is well-known that $\Xi$ is principal if  
$f$ is of degree 2, \'etale or ramified at 2 points\footnote{ The are only two other cases where the $P$ gets a principal polarisation: 
(1) $g(C) = 1$ and $g(C')=2$; and (2) $g(C)=2$ and $f$ is non-cyclic of degree 3.}. The Prym map given by 
$[f:C' \ra C] \mapsto (P, \Xi)$ is defined from a suitable moduli space of coverings of smooth curves to the 
moduli space of $D$-polarised abelian varieties $\cA_{\dim P}^D$.  

The Prym maps considered here have their image in a special loci of the moduli of $D$-polarised abelian 
threefolds $\cA_3^D$. Consider the space  $$\cB^D=\{A\in\cA_3^D : \ZZ_2^3\subset \Aut(A)\}.$$

\begin{lem}\label{lem:cB}
 The locus $\cB^D$ is of dimension 3
 and is contained in the set of abelian varieties that are isogenous to the product of elliptic curves.  
\end{lem}
\begin{proof}
This fact is a consequence of the representation theory of finite abelian groups. The group algebra $\QQ[\ZZ_2^3]$ has 3 central orthogonal idempotents, hence every abelian threefold in $\cB^D$ has to be isogenous to the product of 3 elliptic curves.
This lemma can be seen as a special case of \cite[Prop 1.1]{LR}.
\end{proof}

\begin{rem}\label{rmk:reducible}
Note that in general $\cB^D$ is not irreducible. For any $D$, one of the irreducible components is the locus of products of 3 elliptic curves. 
Consider the locus of smooth non-hyperelliptic genus 3 curves with a $\ZZ_2^2$ subgroup of automorphisms  (see Lemma \ref{lem:nonhyp3}). Their Jacobians are principally polarised abelian threefolds that are not polarised products and their group of automorphisms is $\ZZ_2^3$ (because -1 does not come from automorphisms of a curve), so they form another 
component of $\cB^{(1,1,1)}$.

\end{rem}

\section{Klein coverings: Non-isotropic case}\label{sec:niso}

Let $H$ be a curve of genus 2 and $h:C \ra H$ an \'etale double covering,  so $C$ is of genus 3. Let $\sigma$  denote
the automorphism exchanging two sheets  of the covering and $j$ a lift of the hyperelliptic involution of $H$. 
Let $W=\{w_1, \ldots ,w_6\} \subset H$
be the set of Weierstrass points. The curve $C$ has the action of the Klein group
$\ZZ_2 \times \ZZ_2 = \langle j, \sigma \ :  \  j^2=\sigma^2 =1,  \ j\sigma=\sigma j \rangle$. 
So the curve $C$ is hyperelliptic (cf. Lemma \ref{lem:hypgen3}).
We consider now an \'etale double covering of $C$, $f : \tC \ra C$, where $\tC$ is of genus 5 and let $\tau$
denote the involution exchanging the sheets of the covering. Suppose that  $j$ and $\sigma$ lift to involutions
on $\tC$, which we shall denote with the same letters, in a such way that 
 $$
 \ZZ_2^3 =  \langle j, \sigma, \tau \ :  \  j^2=\sigma^2 =\tau^2=1, j\sigma=\sigma j,  \tau\sigma =\sigma \tau,  \tau
 j=j\tau \rangle.
 $$
 
Let $\eta \in JH[2]\setminus\{0\}$ be the 2-torsion point defining the covering $h$ and $\xi \in
JH[2]\setminus\{0\}$ such that $h^*\xi \in JC[2]\setminus\{0\}$ defines the double covering $f$. Assume that the
group $G = \langle \eta, \xi \rangle $ is non-isotropic. 
The following theorem summarises the results from \cite{bo} in the non-isotropic case.

\begin{thm}\label{non-iso-thm}
The following data are equivalent:

\begin{enumerate}
\item a general $(1,4)$-polarised abelian surface $A$; 
\item a triple $(H,\eta, \xi )$, with $H$ a genus 2 curve and $\eta,\ \xi$ generators of a  non-isotropic Klein
subgroup of $JH[2]$;
\item $\tC$ a hyperelliptic curve of genus 5 with $\ZZ^3_2 \subset Aut(\tC)$;
\item a genus 2 curve $H$ together with the choice of 3 points of the ramification locus of the double covering $H \ra \PP^1$;
\item the set of 6 points on $\PP^1$ with a chosen triple of them, up to projective equivalence (respecting the triple).
\end{enumerate}
\end{thm}

\begin{rem}
There is a unique line bundle of characteristic 0 associated to the $(1,4)$-polarisation, such 
that the $(-1)$-action split its space of sections in symmetric and anti-symmetric theta functions.
The anti-symmetric subspace is one-dimensional. In the first statement of Theorem \ref{non-iso-thm} 
a surface is general if the zero-set of the unique anti-symmetric theta function is a smooth curve of genus 5 
(see \cite{bo} for details).
\end{rem}
 
\begin{proof}
The equivalence $(1) \Leftrightarrow (2)$ follows from  \cite[Theorem 4.12]{bo}. The equivalence $(4) \Leftrightarrow (5)$ follows from the construction of genus 2 curves. The implication $(2) \Rightarrow (3)$
is clear from the construction of $\tC$ and \cite[Theorem 4.7]{bo}. The equivalence $(2) \Leftrightarrow (4)$
follows from the fact that any non-trivial 2-torsion point of $H$ can be written as the difference of two 
(distinct) Weierstrass points. Moreover, $\eta_1=w_i-w_j$, $\eta_2=w_k-w_l$ generate a non-isotropic subgroup of
$JH[2]$ if and only if they share a Weierstrass point, i.e. $|\{w_i, w_j \} \cap \{w_k, w_l \}|=1$, see Remark \ref{rem:Weil}. 

It remains to show the implication $(3) \Rightarrow (2)$. By Lemma \ref{fpf-inv} (with $k=2$)
one can assume that  $\tau$ and $\sigma$ are fix-point-free involutions. So, by Lemma \ref{lem-fpf},  
$\sigma\tau$ is also fix-point-free involution. In conclusion, in the non-isotropic case the three involutions
$\sigma, \tau$ and $\tau\sigma$ are indistinguishable, so $g_{\sigma}=g_{\tau}=g_{\tau\sigma}=3$ and
$H=\tC/ \langle \sigma, \tau \rangle$ is a genus 2 curve. Since $\tC$ is hyperelliptic, the Klein covering has to be non-isotropic. 
\end{proof}

\begin{cor}
Every condition in Theorem \ref{non-iso-thm} defines an irreducible variety. In particular, the moduli space of non-isotropic Klein coverings of genus 2 curves is an irreducible variety.
\end{cor}
\begin{proof}
The set described in statement (5) is the quotient of non-ordered pairs of triples in 
$(\PP^1)^{(3)} \times (\PP^1)^{(3)} \setminus \{\textrm{diagonals} \}$  by the action of $\PP GL(2)$, 
which is clearly  an irreducible variety. 
\end{proof}

\begin{rem}
As pointed out in \cite{bo}, the conditions of Theorem \ref{non-iso-thm} carry natural 'dualisations'. In the
first condition one can take the dual surface. In the fourth and fifth condition one can take the remaining
triple. In the second condition one takes the orthogonal complement to the Klein subgroup with respect to the Weil pairing. %{\color{blue} In particular, one should be able to define duality for hyperelliptic curves of genus 5 with 8 involutions.}
\end{rem}

From now on we switch the notation of $j$ to $\iota$ in order to stress the fact that $\tC$ is hyperelliptic with hyperelliptic involution 
$\iota$.
In \cite[Theorem 5.5]{bo} we have proved that the associated Prym variety $P = P(\tC / H)$ has restricted
polarisation of type $(1,1,4)$ and is isogenous to the product of elliptic curves
$E_{\sigma} \times E_{\tau}  \times E_{\sigma\tau}$  with 
 $$
 E_{\sigma} = \tC/\langle \sigma, \iota\tau \rangle, \quad E_{\tau}  = \tC/\langle \tau, \iota\sigma \rangle, \quad 
 E_{\sigma\tau}= \tC/\langle \sigma\tau, \iota\tau \rangle.
$$
Figure \ref{fig:points-non-isotropic} backtracks the preimages of the three distinguished
points on $\PP^1$ and the ramification points of the maps. 
\begin{figure}[ht] 
    \centering
    \includegraphics[scale=0.8]{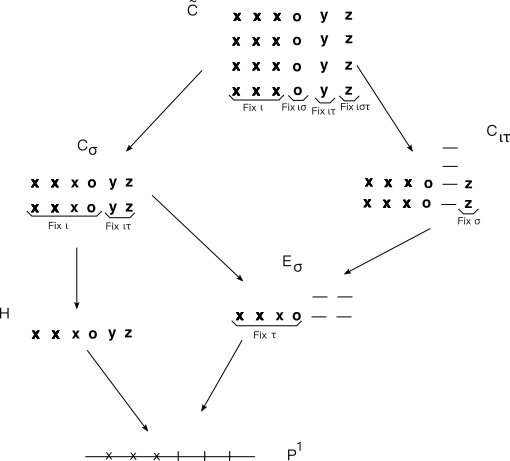}
    \caption{Ramification points (non-isotropic case)}
    \label{fig:points-non-isotropic}
\end{figure}

\begin{rem}\label{rem:startingfrompoints}
Figure \ref{fig:points-non-isotropic} shows how to construct all the curves in the covering starting from 6 points in $\PP^1$. In particular if $H$ is branched at $0,1,\infty,b_1,b_2,b_3$ then the elliptic curves $E_\s,E_\t,E_{\s\t}$ are branched at $0,1,\infty$ and $b_i$.
\end{rem}

In order to show injectivity of the Prym map, we start with some results concerning the Prym variety. 
 
\begin{lem}\label{lem:psi}
 The addition map 
$$
\varphi:  E_{\sigma} \times E_{\tau} \times E_{\sigma\tau} \ra P, \qquad (a,b,c) \mapsto a+b+c,
$$
defines an isogeny of degree 16. Moreover, there is a polarisation $\Xi$ of type $(1,1,4)$ on $P$, such that
$\varphi^*\Xi $ is algebraically equivalent
to four times the principal polarisation on $E_{\sigma} \times E_{\tau} \times E_{\sigma\tau}$.
\end{lem}

\begin{proof}
The fact that $\varphi$ is an isogeny and the restriction of the principal polarisation
$\Xi:= {\Theta_{J\tC}}_{|P}$ is of type $(1,1,4)$ is proven in \cite[Theorem 5.5]{bo}.
Consider the following commutative diagram.

 \begin{equation} \label{iso-diag}
\xymatrix@R=1cm@C=1.5cm{
 E_{\sigma} \times E_{\tau} \times E_{\sigma\tau} \ar[r]^(.65){\varphi} \ar[d]_{\lambda_{\varphi^*\Xi}}  & P \ar[d]^(.4){\lambda_{\Xi}} \\
 \widehat{E}_{\sigma}  \times  \widehat{E}_{\tau} \times  \widehat{E}_{\sigma\tau} & \widehat{P} \ar[l]_(.3){\hat{\varphi}} 
 }
\end{equation}
Clearly, $\Ker \varphi \subset \Ker \lambda_{\varphi^* \Xi}$ and since $\Xi$ is of type $(1,1,4)$,
$\deg \lambda_{\Xi} =16$. In order to compute 
the degree of $\varphi$ we consider the description of the elliptic curves as fixed loci inside of the 
Jacobian of $\tC$:
$$
E_{\sigma} =  \Fix(\sigma,  \iota\tau),  \quad E_{\tau}= \Fix(\tau,  \iota\sigma)  \quad E_{\sigma\tau}=\Fix(\iota\tau, \iota\sigma).
$$ 
Let $(a,b,c) \in \Ker \varphi $, then $c=-a-b$. Applying $\iota\sigma$  and  $\iota\tau$ to $ -a-b$ we get 
$$
\iota\sigma (-a-b) = -\iota a -b  = -a-b \quad \textnormal{ and } \quad \iota\tau (-a-b) = -a - \iota b  = -a-b,
$$
so $\iota a=a$ and $\iota b= b$,  that is, $a,b$ and $c$ are 2-torsion points in their respective elliptic curves. This implies
\begin{equation}\label{eq:kervarphi}
    \Ker \varphi = \{ (a,b, -a-b) \ \mid \  a \in E_{\sigma}[2], \ b \in E_{\tau}[2] \},
\end{equation}
in particular that $|\Ker \varphi | =16$. Since $\deg \hat \varphi =\deg \varphi =16$,  we have $\deg 
\lambda_{\varphi^*\Xi} = 16^3= 4^6 $. This together with the fact $\Ker  \lambda_{\varphi^*\Xi} 
\subset  E_{\sigma}[4] \times E_{\tau}[4] \times E_{\sigma\tau}[4]$, 
implies that $ \lambda_{\varphi^*\Xi} $  is four times the principal polarisation. 

One can also argue that since the elliptic curves are quotients by subgroups of order 4, the respective coverings have to be $4:1$ and hence the restricted polarisations have to be of type $(4)$. Then $\varphi$, as an addition map, becomes a polarised isogeny. 
\end{proof}

Denote by $Z$ the set of 2-torsion points on $E_\s\times E_\t\times E_{\s\t}$. Let $G_P=\varphi(Z) \subset P[2]$. Consider the quotient map $\pi_P:P\lra P/G_P$.
 \begin{lem}\label{lem:mult2}
 There exists an isomorphism $P/G_P=E_\s\times E_\t\times E_{\s\t}$, such that the
 composition $\pi_P\circ\varphi: E_\s\times E_\t\times E_{\s\t}\lra P/G_P$ becomes the multiplication by $2$. In particular, $G_P\simeq \ZZ_2\times\ZZ_2$ is the set of $2$-torsion points that lie in the kernel of  $\lambda_{\Xi}: P \ra \widehat P$.
 \end{lem}
\begin{proof}
Since $|Z|=64$ and $|\ker\varphi|=16$ are 2-groups, we get that $G_P\simeq \ZZ_2\times\ZZ_2$.
Observe that $\ker(\pi_P\circ\varphi)=Z$. If $m_2$ denotes the multiplication by $2$ in $E_\s\times E_\t\times E_{\s\t}$ then the isomorphism theorem gives the following commutative diagram. In subscripts we write the polarisation types.
\begin{equation} \label{diag6}
\xymatrix@R=.7cm@C=.7cm{
& (E_\s\times E_\t\times E_{\s\t})_{(4,4,4)} \ar[dd]_{m_2} \ar[dr]^{\varphi} & \\
& & P_{(1,1,4)}\ar[dl]^{\pi_P}\\
& (E_\s\times E_\t\times E_{\s\t} )_{(1,1,1)}&
}
\end{equation}
Note that by construction all maps are polarised isogenies.
In particular, the kernel of $\pi_P$ is  a subgroup of 
$\ker \lambda_{\Xi}$ and hence $G_P \subset \ker\lambda_\Xi \simeq \ZZ_4\times\ZZ_4$.
\end{proof}

 Denote by $\cR_2^{ni}$ the moduli space parametrising the triples $(H, \eta, \xi)$  with 
 $ \langle \eta, \xi \rangle \simeq V_4 $ a non-isotropic subgroup and 
 $$
 \cB^{ni} :=\cB^{(1,1,4)}=\{(P, \Xi) \in \cA_3^{(1,1,4)} \ \mid \ \ZZ_2^3 \textnormal{ acts on }  P \textnormal{ compatible with }\Xi\}.
 $$
 We define the Prym map for non-isotropic Klein covers as
 $$
 {\Pr}^{ni}:  \cR_2^{ni} \ra \cB^{ni},  \qquad (H, \eta, \xi) \mapsto (P, \Xi).
 $$
  
Observe that the dimension of both moduli space is 3 (cf. Lemma\eqref{lem:cB}).

 \begin{thm}\label{thm:nonisoprym}
 The Prym map ${\Pr}^{ni}:  \cR_2^{ni} \ra \cB^{ni}$ is injective.
 \end{thm}
 \begin{proof}
 We show the statement by explicitly constructing the inverse map.
 
 Let $(P,\Xi)$ be a $(1,1,4)$ polarised abelian threefold that is in the image of the Prym map.
 Note that $G_{P}$ is uniquely defined as the set of 2-torsion points that are in $\ker \lambda_{\Xi}$ and it is of cardinality $4$. Consider the quotient map $\pi$ to a principally polarised threefold $P/G_{P}$. Since $P$ is in the image of the Prym map, by Lemma \ref{lem:mult2}, $P/G_{P}=E_\s\times E_\t\times E_{\s\t}$ with principal product polarisation so, up to permutation, the elliptic curves are uniquely defined.
 For $j\in\{\s,\t,\s\t\}$, let $E_j'$ be the preimage of $E_j$ via $\pi_{P}$. Note that $E_j'\cong E_j$ with $\pi|_{E_j'}$ being multiplication by two. 

Denote the elements of $G_{P}$ by $0,G_0,G_1,G_\infty$. Note that for each $j$ we have that $\ker \pi = G_{P} \subset E_j'$, so $G_{P}=E_j'[2]$. Hence, there exists a unique map $E_j'\lra\PP^1$, such that $G_0,G_1,G_\infty$ are preimages of the branch points $0,1,\infty$ respectively and $0$ is the ramification point of some $b_j\in\PP^1\setminus\{0,1,\infty\}$.

By Figure \ref{fig:points-non-isotropic}, the points $b_j$ are pairwise different from each other, so we have constructed the set of six points in $\PP^1$ with a distinguished triple. Note that renaming the points of $G_{P}$ gives precisely the projective equivalence that respect the triple.

By Theorem \ref{non-iso-thm}, Lemmas \ref{lem:psi} and \ref{lem:mult2}, and Remark \ref{rem:startingfrompoints}, it is obvious that the constructed map is the inverse of the Prym map.
 \end{proof}

Now, we would like to understand the Prym map in terms of abelian varieties. Such characterisation will be needed in Section \ref{sec:pryms}.
Let us introduce the space $\cA_1[2]^{(3)}$ of non-ordered triples 
$$
((E_1, \lambda_1), (E_2, \lambda_2), (E_3, \lambda_3)),
$$
where $E_i$ is an elliptic curve and $\lambda_i: E_i[2] \stackrel{\sim}{\ra} V_4$ is an isomorphism
onto the Klein group $V_4$, called a level $2$-structure, see \cite[Section 8.3.1]{bl}.  Two pairs 
$(E_1, \lambda_1)$ and $(E_2, \lambda_2)$ are isomorphic if there exists
an isomorphism $E_1 \stackrel{\rho}{\lra} E_2$ such that $\lambda_1 = \lambda_2 \circ \rho|_{E_1[2]}$.
Let
$$
\cD:=\{ (E_1, \lambda_1), (E_2, \lambda_2), (E_3, \lambda_3) \ \mid \ (E_i, \lambda_i) \simeq (E_j,
\lambda_j), \ i \neq j\}
$$
and $\cU:= \cA_1[2]^{(3)} \setminus \cD$.

\begin{prop}\label{prop:PhiPsi}
There exist maps  $\Phi$ and $\Psi$ both of degree 6, such that the following diagram commutes
 \begin{equation} \label{non-iso-diag}
\xymatrix@R=1cm@C=1.5cm{
& \cU  \ar[dr]^{\Phi} \ar[dl]_{\Psi}& \\
 \cR_2^{ni} \ar[rr]^{{\Pr}^{ni}}& & \cB^{ni} 
 }
\end{equation}
 
\end{prop}
\begin{proof}
Note that for every couple $(E_i,\lambda_i) \in \cA_1[2]$ one has a well 
 defined map $E_i \ra \PP^1$ ramified at the 2-torsion points
 with $$
 \lambda_i^{-1}(0,1) \mapsto 0, \quad \lambda_i^{-1}(1,1) \mapsto 1
\quad \lambda_i^{-1}(1,0) \mapsto \infty,
 \quad
 \lambda_i^{-1}(0,0) \mapsto b_i.$$
In this way, the cross-ratio $b_i$ is uniquely determined by $(E_i,\lambda_i)$ and the construction gives six points $0,1,\infty,b_1,b_2,b_3\in\PP^1$ with a distinguished triple. Hence, we have constructed a map $\Psi:\mathcal{U}\lra \mathcal{R}_2^{ni}$.

 In order to define $\Phi$, let 
 $$
 K:=\{ (x,y, z) \in E_1[2] \times E_2[2] \times E_3[2] \ \mid \  \lambda_1(x) + \lambda_2(y) + \lambda_3(z) =(0,0) \}
 $$  
 be a subgroup  and  $\pi_K: E_1 \times E_2 \times E_3 \ra  E_1 \times E_2 \times E_3 /K$ the canonical projection.
 Note that $K$ is the kernel of a surjective group homomorphism  $E_1[2] \times E_2[2] \times E_3[2] \ra V_4$,
 so it has cardinality 16. Moreover, $K$ being contained in the set of $2$-torsion points is isotropic with
 respect to the four times the product polarisation.
 We define 
 $$
 \Phi ((E_1, \lambda_1), (E_2, \lambda_2), (E_3, \lambda_3))  = ((E_1 \times E_2 \times E_3) / K, 
 \Xi)
 $$
 where the polarisation $\Xi$ is of type $(1,1,4)$ such that $\varphi^*(\Xi)$ is four times the principal product
polarisation on $E_1 \times E_2 \times E_3$.

Observe that the group of automorphisms
$f_{v}:V_4\ra V_4$, acts on the triples $\{\lambda_i\}_{i=1,2,3}$ by $\{f_{v}\circ\lambda_i\}_{i=1,2,3}$.

This action induces projective equivalences of $\PP^1$ respecting the distinguished triple and also it clearly preserves $K$. Therefore, $\deg \Psi=\deg \Phi = |\Aut(V_4)|= 6$.
The commutativity of Diagram \ref{non-iso-diag} is a straightforward computation.
\end{proof}

\section{Klein coverings: Isotropic case}\label{sec:iso}

Unlike the non-isotropic case, the lift of the hyperelliptic involution to an isotropic Klein 
covering $\tC \ra H$ is not hyperelliptic. This means that we have a few more quotient curves that are not
rational. We will use them to show that the Prym map is injective.

Let $H$ be a genus 2 curve and $\tC \ra H$ be an \'etale isotropic Klein covering of $H$. 
Let $\sigma$ and $\tau$ be the fixed point free involutions interchanging the sheets of the covering.
Consider the genus 3 curves $C_{\sigma}=\tC/\s,C_{\tau}=\tC/\t, C_{\sigma\tau}=\tC/\s\t$. 
Since the hyperelliptic involution  $\iota$ on $H$  lifts to involutions on these curves,
according to Lemma \ref{lem:hypgen3},  $C_{\sigma},C_{\tau},C_{\sigma\tau}$ are all hyperelliptic.

%The lifts of $\iota$ to $\tC$ are called $j,j\t,j\s,j\s\t$.
\begin{prop}\label{prop:involutionj}
Denote the lifts of $\iota$ to $\tC$ to be $j,j\t,j\s,j\s\t$. We can choose $j$ to be fixed point free and the
other involutions to be elliptic, i.e. $g(\tC/j\t)=g(\tC/j\s)=g(\tC/j\s\t)=1$.
\end{prop}
\begin{proof}
Let $i$ be the hyperelliptic involution on $C_{\t}$. Since the lifts of $i$ are either $j,j\t$ or $j\s,j\s\t$ without loss of generality we denote them by $j\s$ and $j\s\t$. Since $|\fix(i)|=8$, we have that $|\fix(j\s)|+|\fix(j\s\t)|=16$.
Since the covering is isotropic and hence the curve $\tC$ is not hyperelliptic, we get that for any involution $\alpha$, $|\fix(\alpha)|<12$.
Using Lemma \ref{lem:lift} for $j\s$ and $j\s\t$, we get that $4<|\fix(j\s)|<12$ is divisible by 4 and therefore $|\fix(j\s)|=|\fix(j\s\t)|=8$.

Now, consider the lifts of the hyperelliptic involution $\iota$ of $C_{\s}$. By construction, again the lifts of $\iota$ are
either a pair $j,j\s$ or $j\t,j\s\t$. Since $j\s$ and $j\s\t$ are indistinguishable, without loss of generality,
we denote the lifts from $C_{\s}$ to be $j\t,j\s\t$. Repeating the argument above, we get that
$|\fix(j\s)|=|\fix(j\s\t)|=|\fix(j\t)|=8$. 

Since the fixed point sets are mutually disjoint (as different lifts of $\iota$) and $|f^{-1}(\fix(\iota))|=4\cdot 6=24$ we get that the fourth lift, denoted by $j$, is fixed point free.
\end{proof}

\begin{rem}
In the non-isotropic case, we have lifted the hyperelliptic involution from $H$ to three hyperelliptic
involutions on $C_\s,C_\t,C_{\s\t}$ and the hyperelliptic involution on $\tC$ is the unique simultaneous lift 
of these three involutions.
On the other hand, in the isotropic case one has to choose the lifts of the hyperelliptic involution from $H$ to 
be three elliptic involutions on $C_\s,C_\t,C_{\s\t}$. Then the involution $j$ is the unique simultaneous lift of
these three involutions.
\end{rem}

\begin{defn}
We define the genus 3 curve $C_j=\tC/j$. One observes that the three involutions $\s,\t,\s\t$ on $\tC$ descend
to $C_j$ since they all commute.
\end{defn}
Define  the elliptic curves 
$$E_\s=\tC/\langle j ,\s\rangle,\quad  E_\t=\tC/\langle j ,\t \rangle ,\quad E_{\s\t}=\tC/\langle j,\s\t \rangle,$$ 
which are also respectively, quotients of $C_\s,\ C_\t,\ C_{\s\t}$  by elliptic involutions. In particular, 
they are quotients of $\tC$ by four elements subgroups of automorphisms and they fit in the following extended
commutative diagram:
\begin{equation} \label{diag1}
\xymatrix@R=1.1cm@C=.9cm{
& \tC \ar[d]_{\text{\'et}} \ar[dr]_{\text{\'et}} \ar[dl]_{\text{\'et}} \ar[drr]^{\text{\'et}}&\\
C_{\sigma} \ar[d] \ar[dr] & C_{\tau}  \ar[dl] \ar[dr]   & C_{\sigma\tau}   \ar[dll]\ar[dr] &\mathbf{C_j} \ar[d] \ar[dl] \ar[dll]\\
H \ar[dr]& E_\sigma  \ar[d]&E_\tau \ar[dl]&E_{\s\t} \ar[dll]\\
\ \ \  \ \  \ \ \ \ \ \  \ \  \ \ \ \ \ &\PP^1&&
}
\end{equation}

\begin{prop}\label{prop:samedivisor}
The curves $C_\s$ and $C_j$ are double coverings of $E_\s$ branched along the same divisor.
Moreover, the curve $C_j$ is non-hyperelliptic. 
\end{prop}

\begin{proof}
The first part follows from the fact that $E_\s=\tC/\langle j,\s \rangle$ is a quotient of both $C_j$ and $C_\s$ 
by respective involutions, and the fact that $\tC\ra C_\s$ and $\tC\ra C_j$ are \'etale.

One can check that the double covering 
$C_j \ra E_{\s}$ is defined by the line bundle $\cO_{E_{\tau\sigma}}(w_1 + w_2) \in
\Pic^2(E_{\s})$,
where $w_1, w_2$ are two Weierstrass points with respect to the involution $\tau$ on
$E_{\sigma}$ (like $oo$ or $zz$ in Figure \ref{fig:points}). 
 Then, by Lemma \ref{lem:doucovell4p}, $C_j$ can not be hyperelliptic.
\end{proof}

\begin{thm}\label{iso-thm}
Every of the following data are equivalent:
\begin{enumerate}
\item 3 pairs of points in $\PP^1$ up to projective equivalence (respecting the pairs);
\item a genus 2 curve $H$ with an isotropic Klein subgroup of $JH[2]$;
\item a non-hyperelliptic genus 5 curve with a Klein subgroup of fixed point free involutions;
\item a non-hyperelliptic genus 3 curve with a $\ZZ_2^2$ subgroup of automorphisms.
%\item a genus 1 curve with a chosen involution, two pairs of its fixed points and 4 points that are pairwise exchanged.
\end{enumerate}
\end{thm}

\begin{proof}
$(1) \Leftrightarrow (2)$ Let $(P_1, P_2), (P_3, P_4), (P_5, P_6)\in\PP^1$ be three pairs of points in $\PP^1$.
Consider the double cover $\pi:H\lra\PP^1$ branched in these points and let
$w_1,\ldots,w_6$, with $\pi(w_i) = p_i$ be the Weierstrass points of $H$. 
By Corollary \ref{cor:weier} the elements $\{0, w_1-w_2, w_3-w_4, w_5-w_6\}\subseteq JH[2]$ form
a Klein subgroup and, by Remark \ref{rem:Weil}, this subgroup is isotropic.
Conversely, every genus 2 curve with an isotropic Klein subgroup  of $JH[2]$ arises in this way.

$(1) \Leftrightarrow (4)$ For each pair of points $(P_i, P_j)$ one can construct the elliptic curve that is a double
covering of $\PP^1$ branched in $P_k$ with $k \in \{1, \ldots, 6\} \setminus\{i,j\}$.
The preimage of the pair $(P_i, P_j)$ becomes the branching divisor of two double coverings,
giving rise to two genus 3 curves, one hyperelliptic and one non-hyperelliptic.
These coverings are illustrated in Figure \ref{fig:points}, by  $C_{\s} \ra E_{\s}$
and $C_{j} \ra E_{\s}$ respectively. By Proposition \ref{prop:samedivisor}, $C_j$ is non-hyperelliptic and it admits 
two commuting involutions, one arising from the double covering $C_j  \ra E_{\s} $ and the other as the lift of the hyperelliptic involution. Conversely,  by Lemma \ref{lem:nonhyp3},  any non-hyperelliptic genus 3 curve 
$C_j$ with a $\ZZ^2_2$ group of automorphisms  is constructed as the $4:1$ covering of $\PP^1$ with the monodromy being the Klein group, and the branching divisor are three pairs of points, which are images of the fixed points of the three involutions.

$(2) \Leftrightarrow (3)$ This equivalence follows from the 
equivalence $(2) \Leftrightarrow (3)$ in Theorem \ref{non-iso-thm}.
\end{proof}

\begin{rem}
The curve $\tC$ is defined as an isotropic Klein covering of $H$ or equivalently, as an \'etale double covering 
of $C_j$ given by the distinguished 2-torsion point $\alpha\in JC_j[2]$, see Corollary \ref{cor:distpoint}. 
\end{rem}

\begin{figure}[h]
    \centering
    \includegraphics[scale=0.8]{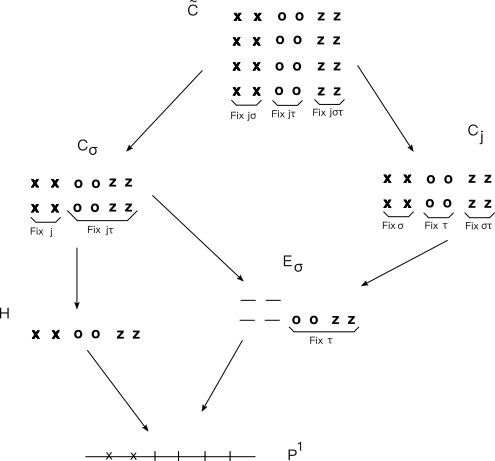}
    \caption{Ramification points 
    (isotropic case)}
    \label{fig:points}
\end{figure}

\begin{cor}
The moduli space of isotropic Klein coverings of genus 2 curves is an irreducible variety.
\end{cor}
\begin{proof}
The set in statement (1) can be described as the quotient of $(\PP^1)^{(2)} \times (\PP^1)^{(2)}\times
(\PP^1)^{(2)} \setminus \{ \textrm{diagonals}\}$ by the diagonal action of $\PP GL(2)$, which is an irreducible 
set.
\end{proof}

\begin{rem}
One finds a similar result to the equivalence of conditions (3) and (4) of Theorem \ref{iso-thm} in \cite[Proposition 3.1]{KMV}. 
There the authors proved a bijection between some genus 5 curves with three bielliptic
involutions and genus 3 curves with three bielliptic involutions with product 1.

\end{rem}

Let $f: \tC  \ra H $ be an isotropic Klein covering of $H$ and $h: \tC \ra C_j $ the double
\'etale map  over $C_j$.
\begin{prop} \label{iso-Prym}
The image of $h^*: JC_j \ra J\tC$ equals the Prym variety $P:=P(\tC/ H)$ as polarised abelian varieties.
In particular, $\Xi$ is of type $(1,2,2)$.
\end{prop}

\begin{proof}
Since $\dim JC_j =\dim P = 3 $, it is enough to show that $h^*(JC_j) \subset 
P= \Ker (1 + \sigma + \tau + \sigma\tau )^{0}$. This follows at once from 
the fact $h^*(JC_j) = \im (1+j) $ and $1 + \sigma + \tau + \sigma\tau + j + j\sigma + j\tau 
+ j\sigma\tau $ is the zero map in $J\tC$. Since $h$ is \'etale, $h^*$ has kernel $\Ker h^*= \langle \alpha \rangle$, with $\alpha \in JC_j[2]\setminus \{0\}$, so
the polarization of $h^*(JC_j)$ is of type $(1,2,2)$.
\end{proof}

\begin{cor}
The map 
$$
\varphi:  E_{\sigma} \times E_{\tau} \times E_{\sigma\tau} \ra P, \qquad (a,b,c) \mapsto a+b+c,
$$
defines an isogeny of degree 16, such that  $\varphi^*\Xi $ is algebraically equivalent
to four times the principal polarisation on $E_{\sigma} \times E_{\tau} \times
E_{\sigma\tau}$.
\end{cor}

\begin{proof}
According to Proposition \ref{prop:samedivisor}, $C_j$ is not hyperelliptic, so the Corollary follows at once 
from Proposition \ref{prop:nonhyp3}.
\end{proof}

Let $\cR^{iso}_2$ be the moduli space parametrising the isotropic Klein coverings over a genus 2 curve and
$$
\cB^{iso} := \cB^{(1,2,2)} =\{ (A, \Xi) \in \cA_3^{(1,2,2)} \  \mid \  \exists \ \ZZ_2^3  \subset  
\Aut(A), \  \textnormal{ compatible with } \ \Xi  \}.
$$
We define the Prym map 
 $$
 {\Pr}^{iso}:  R_2^{iso} \ra \cB^{iso},  \qquad (H, \eta, \xi) \mapsto (P, \Xi).
 $$
As a consequence of Theorem \ref{iso-thm} we have:

\begin{thm}\label{thm:isoprym}
The Prym map ${\Pr}^{iso}: R_2^{iso} \ra \cB^{iso}$ is  injective.
\end{thm}

\begin{proof}
Let $P(\tC /H)$ be an element in the image of $\Pr^{iso}$. According to Proposition \ref{iso-Prym},
$P(\tC /H) \simeq JC_j / \langle \alpha \rangle$, for some 2-torsion point $\alpha$ and by Proposition
\ref{prop:samedivisor}, $C_j$ is a non-hyperelliptic curve with a $\ZZ_2^2$ subgroup of automorphisms.
According to  Theorem \ref{iso-thm} one can recover the curve $H$ and the Klein isotropic subgroup of $JH[2]$
such that the corresponding covering $\tC\ra H$ has $JC_j / \langle \alpha \rangle$ as Prym variety.
\end{proof}

\subsection{A description of double coverings arising in the construction}

We shall now describe explicitly the 2-torsion points defining the \'etale coverings from
$\tC$. Consider the elliptic curve $E_{j\t}=\tC/j\t$, which is also an \'etale double cover of $E_\t= E_{j\t}/
\langle \tau \rangle$. Similarly, we define the elliptic curves 
$E_{j\s}=\tC/j\s$ and $E_{j\s\t}=\tC/j\s\t$. They fit in the following commutative diagrams.

\begin{multicols}{3}
\begin{equation}\label{eq:diagram}
\xymatrix@R=.5cm@C=.5cm{
& \tC \ar[dr]^{\text{\'et}} \ar[dl]& \\
E_{j\s} \ar[dr]^{\text{\'et}}& & C_j \ar[dl]\\
 &E_{\s}& }
\end{equation}

\begin{equation*} 
\xymatrix@R=.5cm@C=.5cm{
& \tC \ar[dr]^{\text{\'et}} \ar[dl]& \\
E_{j\t} \ar[dr]^{\text{\'et}}& & C_j \ar[dl]\\
 &E_{\t}& }
\end{equation*}

\begin{equation*} 
\xymatrix@R=.5cm@C=.5cm{
& \tC \ar[dr]^{\text{\'et}} \ar[dl]& \\
E_{j\s\t} \ar[dr]^{\text{\'et}}& & C_j \ar[dl]\\
 &E_{\s\t}& }
\end{equation*}
\end{multicols}

By abuse of notation, let
\begin{eqnarray*}
\Fix(j\sigma) &=& \{ S_1, S_2, jS_1, jS_2,   \tau (S_1), \tau (S_2),  j\tau (S_1), j\tau (S_2)\} \\
\Fix(j\tau) &=& \{ T_1, T_2, jT_1, jT_2, \sigma (T_1), \sigma (T_2), j\sigma (T_1), j\sigma (T_2)\} \\
\Fix(j\sigma\tau) &=& \{ R_1, R_2, jR_1, jR_2,\tau (R_1), \tau (R_2), j\tau (R_1), j\tau (R_2)\} 
\end{eqnarray*}
the fixed locus  of the automorphisms in $\tC$. The quotient map $\tC \ra E_{j\sigma}$ has 8 
ramification points and one checks that the image of $\Fix(j\sigma\tau)$ are the Weierstrass points of
the elliptic curve $E_{j\sigma}$. Moreover, the  double covering $E_{j\sigma} \ra \PP^1 \simeq E_{j\sigma}/\tau$
gives the linear equivalence
$$
T_1 +\sigma T_1 \sim T_2 + \sigma T_2,
$$
in $E_{j\sigma} = \tC / j\sigma$, where we keep notation for the points in the quotient. 
This implies the linear equivalence
\begin{equation} \label{lin-equiv}
T_1 + jT_1 + \sigma T_1 + j\sigma T_1 \sim T_2 + jT_2 + \sigma T_2 +j \sigma T_2
\end{equation}
in $\tC$. It is not difficult to verify that the 2-torsion point $\eta:= [T_1 +\sigma T_1 - (T_2 + \sigma T_2)]$
in $JC_j$ (again, by abuse of notation, the points in the quotients are denoted with the same symbols),
is in the kernel of the pullback $JC_j \ra J\tC$ and it is different from zero since $C_j$ is non-hyperelliptic.
So $\eta \in JC_j[2]$ defines the \'etale covering $\tC \ra C_j$.  

On the other hand, consider the non-zero 2-torsion point $[T_1 + jT_1 - (T_2 + jT_2)] \in JC_{\sigma} $ which 
according to \eqref{lin-equiv} is in the kernel of $JC_{\sigma} \ra J\tC$,  it is the element in $JC_{\sigma}[2]$
defining the \'etale covering $\tC \ra C_{\sigma}$.  

Using a similar argument one can write down the 
2-torsion points defining the other \'etale coverings.

\subsection{Decomposition of $J\tC$}

Now, we would like to describe the Jacobian of $\tC$ in terms of its subvarieties. For non-isotropic case it has
been done in \cite[Section 5]{bo}. We recall the following notation from \cite{bo}. Let $M_1,\ldots,M_k$ be abelian subvarieties
of an abelian variety $A$ such that the associated rational idempotents $\varepsilon_{M_i}\in End_\QQ(A)$ satisfy
$\Sigma_i\varepsilon_{M_i}=1$. Then we write $A=M_1\boxplus M_2\boxplus\ldots\boxplus M_k$. The advantage of such
notation with respect to writing isogenous is that, if all the $M_i's$ are simple and $Hom(M_i,M_j)=0$ for every 
$i\neq j$, then such presentation is
unique (up to permutation). 

\begin{thm}\label{thm:jtc}
Let $H$ be a general genus 2 curve (i.e. such that $JH$ is simple).
Let $f:\tC\lra H$ be an isotropic Klein covering defined by a group $G$ and let $A=JH/G$. Let $\s,\t,\s\t$ be the covering involutions on $\tC$ and let $j$ be defined by Proposition \ref{prop:involutionj}. 
Then $$J\tC=A\boxplus E_{j\s}\boxplus E_{j\t}\boxplus E_{j\s\t}.$$
In this presentation one can find images of Jacobians of quotient curves and Pryms. For example
$P(\tC/H)=JC_j/\alpha=E_{j\s}\boxplus E_{j\t}\boxplus E_{j\s\t}$, the image of $JC_\s$ is $A\boxplus E_{j\s}$ and its Prym $P(\tC/C_\s)=E_{j\t}\boxplus E_{j\s\t}$.
\end{thm}
\begin{proof}
The proof is straightforward. Since $f$ is \'etale and given by $G$, we get that $J\tC=A\boxplus P(\tC/H)$. By Proposition
\ref{prop:nonhyp3}, $JC_j=E_{\s}\boxplus E_{\t}\boxplus E_{\s\t}$ and the 2-torsion point
$\alpha$ is in the intersection of three elliptic curves. Hence, $P(\tC/H)=JC_j/\alpha=E_{\s}/\alpha\boxplus E_{\t}/\alpha\boxplus E_{\s\t}/\alpha=E_{j\s}\boxplus E_{j\t}\boxplus E_{j\s\t}$.

To be more precise with the last claim, using Diagram \ref{eq:diagram} and \cite[Proposition 12.3.2]{bl} applied for the curve
$E_{j\s}$ and the quotient map $\tC\ra E_{j\s}$, one gets $2(1+j\s)=1+j+\s+j\s$. Analogously,
$2(1+j\t)=1+j+\t+j\t$ and $2(1+j\s\t)=1+j+\s\t+j\s\t$. Adding the equations together and having in mind that 
the sum of all involutions is the zero map, one gets
$$
2(1+j\s+1+j\t+1+j\s\t)=1+j+\s+j\s+1+j+\t+j\t+1+j+\s\t+j\s\t=2+2j.$$
This shows that $\varepsilon_{E_{j\s}}+\varepsilon_{E_{j\t}}+\varepsilon_{E_{j\s\t}}=
\varepsilon_{JC_{j}/\alpha}$ and in particular $\varepsilon_{E_{j\s}}+\varepsilon_{E_{j\t}}+\varepsilon_{E_{j\s\t}}=\varepsilon_{P(\tC/H)}=1-\varepsilon_{A}$.
\end{proof}
\begin{lem}\label{lem:isoisogeny}
The kernel of the addition map \[\varphi:E_{j\s}\times E_{j\t}\times E_{j\s\t}\lra E_{j\s}\boxplus E_{j\t}\boxplus E_{j\s\t}\subset J\tC\] is of the form
\[ \ker \varphi =\{(0,0,0),(x_1,x_2,x_3)\}\]
for some $x_i\neq0,\ i=1,2,3$.
\end{lem}
\begin{proof}
Since the restricted polarisation to the Prym variety is of type $(1,2,2)$ and to the elliptic curves is of type $(2)$, the kernel is of cardinality 2. By definition, the restriction of $\varphi$ to each elliptic curve is an embedding, hence at least two of $x_i$'s are non-zero.
Assume by contradiction (and without loss of generality) that $x_1=0,\ x_2\neq0,\ x_3\neq0$. Then $\varphi(0,x_2,0)=\varphi(0,0,-x_3)$ and hence $E_{j\t}\cap E_{j\s\t}\neq \{0\}$, so the addition map $E_{j\t}\times E_{j\s\t}\lra P(\tC/C_\s)$ is not an isomorphism. But $C_\s$ is hyperelliptic, so by \cite[p. 346]{M}  $P(\tC/C_\s)$ is polarised isogenous to the product $E_{j\t}\times E_{j\s\t}$, a contradiction.
\end{proof}

\begin{rem}
Theorem \ref{thm:jtc} is a more detailed version of \cite[Theorem 6.3.iv]{rr} in the particular case, when the base curve has genus 2, 
the coverings are \'etale and defined by an isotropic Klein group.
This result can also be deduced from the Kani-Rose decomposition theorem \cite{KR}.
\end{rem}

\section{Characterisation of Prym varieties}\label{sec:pryms}

In order to characterise the image of the Prym map, we need to introduce period matrices. Following  the notation in \cite{bl,Bor}, we denote the Siegel space $\ch_g=\{Z\in M(g,g,\CC): Z=\ ^tZ, \Ima Z>0\}$.
We denote the imaginary part of $z \in\CC$ by $z'=\Ima(z)$ and by $Z[i]$, respectively $Z'[i]$, the i-th column vector of a matrix $Z$, respectively its imaginary part.  Let
$D=\mbox{diag}(d_1,d_2,d_3)$ be a possible polarisation type\footnote{The vector $(d_1, \ldots, d_g )$ as well as the matrix $\mbox{diag}(d_1, \ldots, d_g )$
are both called the {\it type of the polarization}}, i.e. $d_1, d_2, d_3$  are positive integers with
$d_1|d_2,\ d_2|d_3$. If $\Lambda$ is a lattice generated by the column vectors of $Z\in\ch_3$ and $D$, then
we denote by $A_{Z,D}=\CC^3/\Lambda$ the corresponding $D$-polarised abelian threefold.
The canonical projection map from the Siegel space to the moduli space $\cA_g^D$ will be denoted by $\pi_{D}:\ch_g\ra
\cA_g^{D}$.
If we define the alternating bilinear form on the column vectors by $\omega(Z'[i],D[i])=d_i$, 
$i=1,2,3$ (and zero otherwise), then
$\omega$ is the imaginary part of the hermitian form defining a polarisation of type $D$ on $A_{Z,D}$, see
\cite[Section 8.1]{bl} for details. 
 Finally, for a homomorphism $f$ of abelian varieties we denote its analytic representation by $F$.

We shall describe the periods matrices of the Prym varieties in 
$\Ima (P^{ni})  \subset \cB^{ni} \subset \cA_3^{D_4}$ with $D_4=\mbox{diag} (1,1,4)$.
Recall from Proposition \ref{prop:PhiPsi} that the map $\Phi:\cU\lra\cB^{ni}$ is given by:
\[ \Phi ((E_1, \lambda_1), (E_2, \lambda_2), (E_3, \lambda_3))  = ((E_1 \times E_2 \times E_3) / K,  \Xi)=:(P,\Xi),\]
hence it can be seen as a rational map from $\cA_{1}[2]^{(3)}$ defined in the complement of the diagonals 
in $\cA_{1}[2]^{(3)}$.
Moreover, for each Prym variety $P$ in the image of $\Pr^{ni}$ there is an isogeny
\[\varphi:E_1\times E_2\times E_3\lra P.\]

Consider the following set in $\ch_3$

\[\cZ_4=\{Z=(z_{ij})\in\ch_3: -2z_{11}=z_{31}, z_{21}=z_{11}+z_{22}-\frac{z_{33}}{4}, -2z_{22}=z_{32}\}.\]

\begin{lem}\label{lem:z4}
Let $z_1,z_2,z_3\in\CC$.
Assume that a lattice  $\Lambda$ is generated by the column vectors of
\[\left[
\arraycolsep=5pt\def\arraystretch{2}
\begin{array}{cccccc}
\dfrac{z_2+z_3}{4}&\dfrac{z_2}{4}&\dfrac{-z_2-z_3}{2}&1&0&0\\
\dfrac{z_2}{4}&\dfrac{z_1+z_2}{4}&\dfrac{-z_1-z_2}{2}&0&1&0\\
\dfrac{-z_2-z_3}{2}&\dfrac{-z_1-z_2}{2}&z_1+z_2+z_3&0&0&4\\
\end{array}\right]=[Z\ D_4],\]
where the first three columns %, called $Z[1],Z[2],Z[3]$ form a positive definite symmetric 
form a matrix $Z\in\cZ_4$.
Then the abelian threefold $A_{Z,D_4}= \CC^3/\Lambda$ contains three embedded elliptic curves $f_i: E_i \hookrightarrow A_{Z,D_4}$, $i=1,2,3$. The corresponding lattices and analytic representations of $f_i$ are given by:
\begin{align*}
&\Lambda_{E_1}\text{ generated by }[z_1\ 4],\ F_{1}(x)=(0,-\tfrac{x}{2},x)\in \CC^3,\\% A_\Lambda\\
&\Lambda_{E_2}\text{ generated by }[z_2\ 4],\ F_{2}(x)=(\tfrac{x}{2},\tfrac{x} {2},-x)\in \CC^3,\\%A_\Lambda\\
&\Lambda_{E_3}\text{ generated by }[z_3\ 4],\ F_{3}(x)=(-\tfrac{x}{2},0,x)\in \CC^3.%A_\Lambda
\end{align*}
In particular, there exists a polarised isogeny $f_1+f_2+f_3:E_1\times E_2\times E_3\lra A_{Z,D_4}$.
\end{lem}
\begin{proof}
Since $\det(Z')=\frac{1}{4}z_1'z_2'z_3'$, we see that the imaginary parts of $z_i$ are non-zero, hence
$E_i$ are well-defined elliptic curves.

Since $F_{1}(z_1)=2Z[1]+Z[3]$ and $F_{1}(4)=-2D_4[2]+D_4[3]$ are independent primitive vectors in $\Lambda$, we see that $f_{1}$ is an embedding. Moreover, $\omega(2Z'[1]+Z'[3],-2D_4[2]+D_4[3])=4$, hence the restricted polarisation is of type $(4)$.

Similarly, $F_{2}(z_2)=2Z[1]+2Z[2]+Z[3]$ and $F_{2}(4)=2D_4[1]+2D_4[2]-D_4[3]$ shows that $f_{2}$ is an embedding and 
since $\omega(2Z'[1]+2Z'[2]+Z'[3],2D_4[1]+2D_4[2]-D_4[3])=4$ the restricted polarisation is also of type $(4)$.

Lastly, since $F_{3}(z_3)=2Z[2]+Z[3]$ and $F_{3}(4)=-2D_4[1]+D_4[3]$ are linearly independent  $f_{3}$ is an embedding and one checks that the restricted polarisation is of type $(4)$.
\end{proof}

\begin{lem}\label{lem:kerz4}
Let $Z \in\cZ_4$ be a period matrix as in 
Lemma \ref{lem:z4}. The kernel of the addition map \[f_{1}+f_{2}+f_{3}:E_1\times E_2\times E_3\lra A_{Z,D_4}\] is given by sixteen 2-torsion points generated by the images of the following points \[\{(\tfrac{z_1}{2},\tfrac{z_2}{2},0),(0,\tfrac{z_2}{2},\tfrac{z_3}{2}),(2,2,0),(0,2,2)\}.\]
\end{lem}
\begin{proof}
Since some linear combinations of $F_{i}(x_i)$ give the basis vectors of $\CC^3$ we get that the addition map is an isogeny. By construction, it is a polarised isogeny between a $(4,4,4)$ and a $(1,1,4)$ polarised threefold, hence its degree equals 16. Now, it is enough to show that the following 2-torsion points lie in the kernel, but this is a straightforward computation. In particular, one checks that  $F_{1}(2)+F_{2}(2)=(0,1,0)=D_4[2]\in\Lambda$, $F_{2}(2)+F_{3}(2)=D_4[1]\in\Lambda$
and $F_{1}(\frac{z_1}{2})+F_{2}(\frac{z_2}{2})=2Z[1]+Z[2]+Z[3]\in\Lambda$, $F_{2}(\frac{z_2}{2})+F_{3}(\frac{z_3}{2})=Z[1]+2Z[2]+Z[3]\in\Lambda$.
\end{proof}

We define the map ${\widetilde{\Phi}}(z_1,z_2,z_3) = Z \in \cZ_4 $ as in Lemma \ref{lem:z4},  for any $(z_1,z_2,z_3) \in \ch_1^3$ and consider the following diagram
\begin{equation}\label{diag:non-iso} 
\xymatrix@R=1cm@C=1.5cm{
 \ch_1^3 \ar[r]^(.45){\widetilde{\Phi}} \ar[d]_{\pi_{(4)}^{\times3}}  & \cZ_4\subset \ch_3 \ar[d]^(.4){\pi_{(1,1,4)}}  \\
 \cA_{1}[2]^{(3)} \ar@{-->}[r]^(.5){\Phi} &  \overline{\Pr^{ni}(\cR_2^{ni})}
 }
\end{equation}
We have to take the closure of the image of the Prym map in $\cB^{ni}$, since not every
triple of elliptic curves will induce 6 points on $\PP^1$ that enable us to construct the Prym variety. 

\begin{lem}\label{lem:z4inv}
The map $\pi_{(4)}^{\times3}$ in Diagram \ref{diag:non-iso} is surjective. Moreover, $\widetilde{\Phi}$ is a lift of $\Phi$, so that Diagram \ref{diag:non-iso} commutes on the preimage of $\cU$. 
\end{lem}
\begin{proof}
The surjectivity of $\pi_{(4)}^{\times3}$ follows from the fact that 
$\ch_1$ is the universal cover of the moduli space $\cA_1[2]$ of elliptic curves with a level-2 structure (see \cite[Section 8.3.1]{bl}).
Let $((E_1, \lambda_1), (E_2, \lambda_2), (E_3, \lambda_3))$ be an element in $\cU$ and $V_4=\{(0,0),(1,0),(0,1),(1,1)\}$. 
For each elliptic curve $E_i$ one can choose periods, $[z_i \ 1]$ with $z_i\in \ch_1$ such that
$\lambda_i([\tfrac{1}{2}])=(0,1),\ \lambda_i([\tfrac{z_i}{2}])=(1,0)$,
where the bracket denotes the class in the quotient $E_i = \CC/\Lambda_{E_i}$. This condition 
will guarantee that $\widetilde \Phi$ is a lift of $\Phi$. Then, by Lemma
\ref{lem:kerz4}, the kernel 
$K$ of the isogeny $\varphi$ is generated by the images of
\[\{(\tfrac{z_1}{2},\tfrac{z_2}{2},0),(0,\tfrac{z_2}{2},\tfrac{z_3}{2}),(2,2,0),(0,2,2)\}.\]
and $\Phi(E_1,E_2,E_3)=(E_1\times E_2\times E_3)/K$.

In particular, according to Lemma \ref{lem:z4}, one can choose an analytic representation of $\varphi$ to be $(F_{1}+F_{2}+F_{3})$.
Observe that $\widetilde{\Phi}$ is a well defined lift, since the Siegel spaces are simply connected and by construction,  $\Phi\circ\pi_{(4)}^{\times3}=\pi_{D_4}\circ\widetilde{\Phi}$.
\end{proof}

\begin{thm}\label{thmz4}
The locus 
\[\cZ_4=\{Z=(z_{ij})\in\ch_3: -2z_{11}=z_{31}, z_{21}=z_{11}+z_{22}-\frac{z_{33}}{4}, -2z_{22}=z_{32}\}\]
is irreducible of dimension 3. The image of the map $\pi_{D_4}|_{\cZ_4}$ is the closure of the image of $\Pr^{ni}$. In particular, we have found period matrices for the threefolds in the image of $\Pr^{ni}$.
\end{thm}
\begin{proof}
Since the equations defining $\cZ_4$ are linear and independent, $\cZ_4$ is irreducible of dimension 3 and one can choose coordinates, such that each $Z\in\cZ_4$ is of the form satisfying Lemma \ref{lem:z4}. 

The result follows from the commutativity of Diagram %\ref{non-iso-diag} and 
\ref{diag:non-iso}. 
\end{proof}

The isotropic case is quite similar.  Let $D_2=\diag(1,2,2)$ and consider the locus 
\[\cZ_2=\{Z=(z_{ij})\in\ch_3: z_{33}=2z_{31}, z_{32}=0, z_{22}=2z_{21}\}.\]

\begin{lem}\label{lem:z2}
Let $z_1,z_2,z_3\in\CC$.
Assume that $\Lambda$ is generated by the column vectors of
\[\left[
\arraycolsep=5pt\def\arraystretch{2}
\begin{array}{cccccc} 
\dfrac{2z_1+z_2+z_3}{4}&\dfrac{z_2}{2}&\dfrac{z_3}{2}&1&0&0\\
\dfrac{z_2}{2}&z_2&0&0&2&0\\
\dfrac{z_3}{2}&0&z_3&0&0&2\\
\end{array}\right]=[Z\ D_2],\]
where the first three vectors 
form a matrix $Z\in\cZ_2$.
Then the abelian threefold $A_{Z, D_4}$ contains three embedded elliptic curves $E_1, E_2, E_3$ given by the following lattices and analytic representations:
\begin{align*}
&\Lambda_{E_1}\text{ generated by }[z_1\ 2],\ F_{E_1}(x)=(\tfrac {x}{2},0,0)\in \CC^3,\\
&\Lambda_{E_2}\text{ generated by }[z_2\ 2],\ F_{E_2}(x)=(\tfrac {x}{2},x,0)\in \CC^3,\\
&\Lambda_{E_3}\text{ generated by }[z_3\ 2],\ F_{E_3}(x)=(\tfrac {x}{2},0,x)\in \CC^3.
\end{align*}
\end{lem}
\begin{proof}
%The proof is similar to the proof of Lemma \ref{lem:z4}.
Firstly, $\det(Z')=\frac12z_1'z_2'z_3'$ shows that elliptic curves are well defined.
Since $F_{E_1}(z_1)=2Z[1]-Z[2]-Z[3]$ and $F_{E_1}(2)=D_2[1]$ are primitive independent vectors in $\Lambda$, and $\omega(2Z'[1]-Z'[2]-Z'[3],D_2[1])=2$, $E_1$ is embedded in $A_{[Z,D_2]}$ and the restricted polarisation is of type $(2)$.
Analogously, $F_{E_2}(z_2)=Z[2]$ and $F_{E_2}(2)=D_2[1]+D_2[2]$ are linearly independent in $\Lambda$
as well as $F_{E_3}(z_3)=Z[3]$ and $F_{E_3}(2)=D_2[1]+D_2[3]$. Then $E_2$ are $E_3$ are embedded in $A_{Z, D_4}$ and it is 
no difficult to check that the restricted polarisations are both of type $(2)$.
\end{proof}

\begin{lem}
The kernel of the addition map is given by $\{(0,0,0),(\tfrac{z_1}{2},\tfrac{z_2}{2},\tfrac{z_3}{2})\}.$
\end{lem}
\begin{proof}
The analytic representation of the addition map is surjective, hence the map is an isogeny. By construction, it is a polarised isogeny between a $(2,2,2)$ and a $(1,2,2)$ polarised threefold, hence it is of degree 2. Clearly, $F_{E_1}(\frac{z_1}{2})+F_{E_2}(\frac{z_2}{2})+F_{E_3}(\frac{z_3}{2})=Z[1]\in\Lambda$,  which proves the lemma. 
\end{proof}

In order to make computations straightforward, we do not use the characterisation of Prym as
the quotient of the Jacobian $JC_j$, but the fact that $P(\tC/H)=E_{j\s}\boxplus E_{j\t}\boxplus
E_{j\s\t}$, for some elliptic curves. Let $\cA_{1,\ZZ_2}$ denote the moduli space of elliptic curves with a
chosen non-zero  $2$-torsion point. Using Lemma \ref{lem:isoisogeny}, we define
$\varPhi:\cA_{1,\ZZ_2}^3\lra\overline{\Pr^{iso}(\cR^{iso})}$, by
\[\varPhi((E_1,x_1),(E_2,x_2),(E_3,x_3))=E_1\times E_2\times E_3/\{(0,0,0),(x_1,x_2,x_3)\}\]
and consider the diagram
\begin{equation} \label{diag:iso-case} 
\xymatrix@R=1cm@C=1.5cm{
 \ch_1^3 \ar[r]^(.45){\widetilde\varPhi } \ar[d]_{\pi_{(2)}^{\times3}}  & \cZ_2\subset \ch_3 \ar[d]^(.4){\pi_{(1,2,2)}}  \\
\cA_{1,\ZZ_2}^{(3)}  \ar[r]^(.5){\varPhi} &  \overline{\Pr^{iso}(\cR_2^{iso})}
 }
\end{equation}
where $\widetilde\varPhi (z_1,z_2,z_3)= Z \in \cZ_2$ as in Lemma \ref{lem:z2}. We take the closure of the image of the Prym map in $B^{iso}$ because not every triple of elements in $\cA_{1,\ZZ_2}$ gives rise to a non-hyperelliptic curve $C_j$ like in Theorem \ref{iso-thm}
\begin{lem}
The map $\pi^{\times3}_{(2)}$ in Diagram \ref{diag:iso-case} is surjective. 
Moreover, $\widetilde{\varPhi}$ is a lift of $\varPhi$ so that the diagram commutes.
\end{lem}
\begin{proof}
The surjectivity of $\pi^{\times3}_{(2)}$ follows from the fact that $\ch_1$ is a universal cover 
for the moduli of  elliptic curves with a chosen point. In order to make the diagram commutative, for
each elliptic curve $E_i$
we choose periods $[z_i \ 1]$, such that the chosen point is $x_i=[\tfrac{z_i}{2}]$, where the bracket denotes the image of $\tfrac{z_i}{2}$ in the quotient $\CC/\Lambda_{E_i}$. 
\end{proof}

\begin{thm}\label{thmz2}
 The image of the Prym map $\Pr^{iso}$ is dense in the image of the map $\pi_{D_2}(\cZ_2)$.
\end{thm}
\begin{proof}
The locus $\cZ_2$ is irreducible of dimension 3 and obviously contains the image of the Prym map by the comutativity of Diagram \ref{diag:iso-case}.
\end{proof}

\begin{rem}
One could have defined the closure of the image of the Prym map in a different and more adapted subspace of $\cA_3^{D}$.
Our choice has been to define $\cB^{ni}$ and $\cB^{iso}$ but Remark \ref{rmk:reducible} shows that they are reducible.
Another way could be to fix the restricted polarisation types to the elliptic curves in the following way
\[\cE^{D}_{b}=\{(A, \Xi)\in\cA_3^{D}: A=E_1\boxplus E_2\boxplus E_3, \Xi|_{E_i} \text{ is of type } (b),\text{ for } i=1,2,3\},\]
where $D=(d_1, d_2, d_3)$ is the polarization type and $d_3 | b$.
Unfortunately, this locus is reducible even in the `easiest' case $\cE_{2}^{D_2}$. To see this, 
consider a $(1,2)$-polarised abelian surface $S$ that contains elliptic curves $E_1, E_2$ with restricted polarisation of type $(2)$. Then $S\times E_3$ with a $(1,2,2)$ product polarisation  
is an element of $\cE_{2}^{D_2}$ and there is a 3-dimensional family of such products. They are clearly different from Pryms of isotropic Klein coverings.

The discrete invariant that distinguishes products from Pryms are the projections of  $\ker(f_1+f_2+f_3)
\subset E_1 \times E_2 \times E_3$  to each elliptic curve $E_i$. In the Prym case, all the projections $p_i(\ker(f_1+f_2+f_3)\subset E_i$ consists of two points, whereas in the product case, $p_3(\ker(f_1+f_2+f_3))=\{0\}$.
The necessity of this additional invariant also follows from \cite[Section 6]{Guerra}.
\end{rem}

\end{document}